\documentclass[12pt,reqno]{amsart}

\usepackage[english]{babel}
\usepackage[utf8]{inputenc}
\usepackage{amsmath}
\usepackage{amssymb}
\usepackage{amsthm}
\usepackage{enumerate}
\usepackage{graphicx}
\usepackage{fullpage}

\newtheorem{thm}{Theorem}[section]
\newtheorem{lemma}[thm]{Lemma}
\newtheorem{cor}[thm]{Corollary}
\newtheorem{proposition}[thm]{Proposition}
\newtheorem{claim}[thm]{Claim}
\newtheorem{conjecture}[thm]{Conjecture}

\newcommand{\N}{\mathbb{N}}
\newcommand{\R}{\mathbb{R}}

\newcommand{\Case}[2]{\noindent {\bf Case #1:} \emph{#2}}

\newcommand{\mI}{\mathcal{I}}

\newcommand{\mP}{\mathcal{P}}

\newcommand{\sm}{\setminus}

\newcommand{\rarr}{\rightarrow}

\newcommand{\norm}[1]{\left\lVert#1\right\rVert}

\title{Maximum number of colourings. I. 4-chromatic graphs}

\author{Fiachra Knox and Bojan Mohar}

\date{\today}

\begin{document}

\begin{abstract}
It is proved that every connected graph $G$ on $n$ vertices with $\chi(G) \geq 4$ has at most $k(k-1)^{n-3}(k-2)(k-3)$ $k$-colourings for every $k \geq 4$.
Equality holds for some (and then for every) $k$ if and only if the graph is formed from $K_4$ by repeatedly adding leaves.
This confirms (a strengthening of) the $4$-chromatic case of a long-standing conjecture of Tomescu [Le nombre des graphes connexes $k$-chromatiques minimaux aux sommets \'{e}tiquet\'{e}s, C. R. Acad. Sci. Paris 273 (1971), 1124--1126]. Proof methods may be of independent interest. In particular, one of our auxiliary results about list-chromatic polynomials solves a recent conjecture of Brown, Erey, and Li.
\end{abstract}

\maketitle

\section{Introduction}

Let $x$ be a positive integer. By an \emph{$x$-colouring} we mean a function $f:V(G)\to\{1,\dots,x\}$ such that $f(u)\ne f(v)$ whenever $uv\in E(G)$. Note that permuting the colours used in a colouring gives a different colouring. The \emph{chromatic polynomial\/} $P_G(x)$ is the polynomial of degree $n=|V(G)|$ such that the value $P_G(x)$ is equal to the number of $x$-colourings of $G$ for every positive integer $x$.

The chromatic polynomial and its 2-variable version -- the Tutte polynomial -- play an important role in combinatorics. They have applications in theoretical physics, in knot theory, etc. However, many of very basic questions about chromatic polynomials remain unresolved and poorly understood. This paper touches one of such basic elusive problems and provides some new tools that may be of more general interest.

Chromatic polynomials of graphs were introduced by Birkhoff \cite{Bi12} in an attempt to attack the Four-Colour Problem by analytic means. He proved in \cite{Bi30} that $P_G(x)\ge x(x-1)(x-2)(x-3)^{|G|-3}$ for every planar graph $G$ and every $x\ge 5$. In \cite{BiLe46} he conjectured together with Lewis that the same holds for every $x\ge 4$. The Birkhoff-Lewis Conjecture has since been resolved for $x=4$ (the Four-Colour Theorem), but is still open for non-integral values between 4 and 5.

Maximizing the number of colourings within certain graph families has various applications. Wilf \cite{Wilf84} (see also \cite{BeWi85}) came to the problem of maximizing the number of colourings over the set of all graphs with the given number of vertices and edges in algorithm design and analysis. Linial \cite{Linial86} came to the same sort of questions from a computational complexity problem and proposed the same conjecture as Wilf which graphs ought to attain the maximum. In fact, Linial's question was related to maximizing the value $|P_G(-1)|$, which is known to be related to the number of acyclic orientations of $G$, but it turns out that the maximum occurs at the same set of graphs that maximize number of colourings. Partial solutions were obtained by Lazebnik et al. \cite{La90,La91,LaPiWo07}, Dohmen \cite{Do93,Do98}, Simonelli \cite{Si08} and others. Their results culminated in a breakthrough by Loh, Pikhurko, and Sudakov \cite{LoPiSu10} and a surprising follow-up by Ma and Naves \cite{MaNa15}, who confirmed the conjecture by Wilf and Linial asymptotically, but proved that it is false for infinitely many intermediate values.

In theoretical physics and in the theory of graph limits \cite{Lovasz-book}, there is continuing interest in maximizing the number of colourings and more general graph homomorphisms into a fixed target graph. We refer to Engbels and Galvin \cite{En15,EnGa17} for some recent developments.

It is easy to see that for every $k\ge1$ and every integer $x\ge k$, every connected $n$-vertex graph containing a clique of order $k$ has at most
\begin{equation}
   x\cdot(x-2)(x-3)\cdots (x-k+1)\cdot (x-1)^{n-k+1} = x^{\underline{k}}\,(x-1)^{n-k}
\label{eq:1}
\end{equation}
$x$-colourings, where $x^{\underline{k}} = x(x-1)(x-2)\cdots (x-k+1)$ is the falling factorial. This bound is attained for every $x$ if $G$ can be obtained from the $k$-clique $K_k$ by growing an arbitrary tree from each vertex of the clique. In 1971, Tomescu \cite{To71} conjectured that (\ref{eq:1}) is an upper bound for the number of $k$-colourings of any connected $k$-chromatic graph, whether it contains a $k$-clique or not, as long as $k\ge4$:

\begin{conjecture}[Tomescu, 1971]
Let\/ $G$ be a connected $k$-chromatic graph with $k\ge4$. Then $G$ has at most
\begin{equation}
   k!(k-1)^{|G|-k}
\label{eq:2}
\end{equation}
$k$-colourings. Moreover, the extremal graphs are precisely the graphs obtained from $K_k$ by adding trees rooted at each vertex of the clique.
\end{conjecture}

The requirement that $k\ne3$ is necessary since odd cycles (and odd cycles with added trees) have more colourings than specified by (\ref{eq:1}); see \cite{To_DM72,To_JGT94,To_DM97} for more details; and in the case of bipartite graphs ($k=2$), any connected bipartite graph attains the bound.

Tomescu proved \cite{To_JGT90} that all 4-chromatic planar graphs satisfy his conjecture. For the same class of graphs, he proved a stronger conclusion for the number or $x$-colourings (for every $x\ge4$), where the bound of the conjecture is replaced by (\ref{eq:1}). Apart from this achievement, only sporadic results are known \cite{BE15,BEL16,Erey16}. We refer to \cite[Chapter~15]{DKT_Book_05} for additional overview of the results in this area.

In this paper we prove the Tomescu Conjecture for $k=4$ and its extended version for $x$-colourings.
Our main result states the following.

\begin{thm}
\label{thm:main}
Let $G$ be a connected 4-chromatic graph and $x\ge 4$ be an integer. Then
\begin{equation}
   P_G(x) \le x^{\underline{4}}(x-1)^{n-4}.
\label{eq:3}
\end{equation}
Moreover, equality holds for some integer $x\ge 4$ if and only if $G$ can be obtained from $K_4$ by adding a tree on each vertex of $K_4$ (in which case equality holds for every $x \in \R$).
\end{thm}

In the core of our proofs, we use novel counting arguments about precolouring extensions. Formally, we introduce the \emph{list chromatic polynomials} which count the number of extensions. One of our results considers which list assignments allow for the largest number of extensions. Our Corollary \ref{cor:list colourings} describes the extremal cases completely in the case of bipartite graphs. A special case of this kind of problem where one asks about precolouring extensions with one colour used in the neighbourhood of each uncoloured vertex was treated in \cite{BEL16}. Our result in particular solves a conjecture of Brown, Erey, and Li from \cite[Conjecture~3.1]{BEL16}.

In a forthcoming work \cite{AKM2}, we prove the Tomescu conjecture for $k=5$, but the proof becomes more complicated and needs a support of computer calculations. Finally, we treat the general case in \cite{KM3}.

\section{Preliminaries}

We will use standard graph theory terminology and notation as used by Diestel \cite{Diestel2010} or Bondy and Murty \cite{BondyMurty2008}. In particular, we use $n=|G|=|V(G)|$ to denote the order of $G$. The minimum vertex degree of $G$ is denoted by $\delta(G)$. By $N(v)$ we denote the set of neighbours of a vertex $v$, and we use $\chi(G)$ for the chromatic number. We say $G$ is \emph{$k$-chromatic} if $\chi(G)=k$. The graph is \emph{vertex-critical} (\emph{edge-critical}) if the removal of any vertex (edge) decreases the chromatic number. We will frequently use the fact that identifying non-adjacent vertices of a graph $G$ results in a graph $G'$ with $\chi(G') \geq \chi(G)$.

For a vertex-set $U\subseteq V(G)$, $G[U]$ is the subgraph of $G$ induced by $U$. A \emph{$2$-induced subgraph} in a graph $G$ is an induced subgraph $F$ such that $|N(v) \cap V(F)|~\leq 1$ for every $v \in V(G) \sm V(F)$.

If $x$ is a positive integer, we let $[x] = \{1,2,\dots,x\}$.

Let $G$ be a graph and $P_G(x)$ its chromatic polynomial. Throughout the paper we will use the related indeterminate $y=x-1$ and the \emph{shifted chromatic polynomial}:
$$
   Q_G(y) = P_G(y+1).
$$

We say that $G$ is obtained from a graph $H$ by \emph{appending a tree} $T$ if $H$ and $T$ are disjoint and $G$ is formed by identifying a vertex of $H$ with a vertex of $T$.

\begin{lemma}
\label{lem:subgraph}
Let $F$ be a graph and let $G$ be a connected graph containing $F$ as a subgraph. Then
\begin{equation}
   Q_G(y) \leq y^{|G| - |F|}Q_F(y) \label{eq:Lemma2.1}
\end{equation}
for any $y \in \N$. Moreover, if $Q_F(y) > 0$, $y\ge2$, and $F$ does not have two vertices which have the same colour in every $(y+1)$-colouring, then equality holds in $(\ref{eq:Lemma2.1})$ if and only if $Q_F(y) = Q_{G[V(F)]}(y)$ and $G$ is formed from $F$ by appending vertex-disjoint trees to the vertices of $F$.
\end{lemma}

\proof It suffices to show that there are at most $y^{|G| - |F|}$ ways to extend an arbitrary colouring $c$ of $F$ to a colouring of $G$.
To see this, we first observe that, since $G$ is connected, there is an ordering of $U = V(G) \sm V(F)$ such that every vertex $v\in U$ is adjacent either to a vertex of $F$ or to an earlier vertex in $U$.
We remark that either some vertex in $U$ has two neighbours, each of which is either in $V(F)$ or earlier in the ordering, or $G - V(F)$ is a forest and for every component of the forest, there is precisely one edge joining it to $F$. In the latter case, $G$ can be obtained from $F$ by appending vertex-disjoint trees to the vertices of $F$.

Now we can obtain any extension of any colouring $c$ of $F$ to $G$ by colouring the vertices of $U$ in the order discussed above;
then there are at most $y$ colours available at each vertex and thus at most $y^{|G| - |F|}$ extensions in total. This proves (\ref{eq:Lemma2.1}) and also justifies the last statement of the lemma if $Q_F(y) > Q_{G[V(F)]}(y)$.

To prove the `moreover' part of the lemma, suppose that there is a vertex $v \in U$ with two neighbours $v'$ and $v''$, each of which is either in $V(F)$ or earlier in the ordering. From the end of the previous paragraph, we may also assume that $Q_F(y) = Q_{G[V(F)]}(y)$. We claim that in this case the inequality is strict.
If $v', v'' \in V(F)$, consider extensions of a colouring $c$ of $F$ in which $c(v) \neq c(v')$;
then there are at most $y-1$ ways to colour $v$ and the inequality is strict.
So we may assume without loss of generality that $v' \notin V(F)$.
Now when extending $c$ to $G$, either there are at most $y-1$ ways to colour $v'$, or there is a choice of colour for $v'$ (since $y\ge2$) such that there are at most $y-1$ ways to colour $v$.
In either case the inequality is strict. This proves the `moreover' part of the lemma and completes the proof.
\endproof

The following generalization of Lemma \ref{lem:subgraph} will also be used.

\begin{lemma}
\label{lem:decomposition}
Let $H_1,\dots,H_r$ $(r\ge1)$ be subgraphs of a connected graph $G$ such that for each $i=1,\dots,r$ at most one vertex of $H_i$ is contained in $\cup_{j\ne i}H_j$. Let $n=|G|$, $n'=\sum_{i=1}^r |H_i|$. Then
$$
    Q_G(y) \le y^{n-n'} \left(\frac{y}{y+1}\right)^{r-1}\ \prod_{i=1}^r Q_{H_i}(y).
$$
\end{lemma}

\begin{proof}
We may assume that every vertex of $G$ is in some subgraph $H_i$.
If not, then we just add particular vertices as new subgraphs.
This reduces the exponent in $y^{n-n'}$ by one, but also increases $r$ by $1$ and adds a factor of $(y+1)$ in the product,
resulting in no overall change to the bound.

First, we reorder the subgraphs in such a way that for each $i=2,\dots,r$, the subgraph $H_i$ contains a vertex $v_i$ that is either contained in the union $G_{i-1} = \cup_{j=1}^{i-1} H_j$ of previous subgraphs or is adjacent to a vertex in $G_{i-1}$. Since $G$ is connected and subgraphs contain all vertices of $G$, such an ordering exists. Let $u_i$ be a vertex in $G_{i-1}$ that is equal to $v_i$ (if $H_i\cap G_{i-1}\ne \emptyset$) or adjacent to $v_i$ (if $H_i\cap G_{i-1} = \emptyset$).

For the proof, it suffices to prove that any $(y+1)$-colouring of $G_{i-1}$ can be extended to $G_i$ in at most $Q_{H_i}(y)/(y+1)$ ways if $H_i$ intersects $G_{i-1}$ and in at most $yQ_{H_i}(y)/(y+1)$ ways if $H_i$ is disjoint from $G_{i-1}$. To see this observe that the colourings of $H_i$ can be partitioned into $y+1$ classes, all of the same cardinality $Q_{H_i}(y)/(y+1)$, where each class contains the colourings for which the colour of $v_i$ is the same. Now, if $v_i=u_i$, then any colouring of $G_{i-1}$ can only be extended by the colouring class for which the colour of $v_i$ is as in the colouring of $G_{i-1}$. In the case when the graphs are disjoint, the colour of $u_i$ cannot be used for $v_i$, thus at least one out of $y+1$ classes cannot be used for extensions. This proves our claim.
\end{proof}

\begin{lemma}
\label{lem:paths}
Let $G$ be a path of length $r$ with endpoints $v_1$ and $v_2$, and let $c_1, c_2 \in [y+1]$.
Then the number of $(y+1)$-colourings of $G$ in which $v_i$ has colour $c_i$ for $i = 1, 2$ is
$(y^r + (-1)^{r+1})/(y+1)$ if $c_1 \neq c_2$ and $(y^{r} + (-1)^{r}y)/(y+1)$ if $c_1 = c_2$.
\end{lemma}

\proof The proof is by induction on $r$. The cases when $r=0$ or $r=1$ are trivial.
Suppose that $r\ge2$ and that the lemma holds for $r-1$.
Let $v'_2$ be the unique neighbour of $v_2$ in $G$.
If $c_1 \neq c_2$, then by considering the possible colours of $v'_2$ we obtain that the number of $(y+1)$-colourings in which $v_i$ has colour $c_i$ for $i = 1,2$ is
$$
 \frac{y^{r-1} + (-1)^{r-1}y}{y+1} + \frac{(y-1)(y^{r-1} + (-1)^{r})}{y+1} = \frac{y^r + (-1)^{r+1}}{y+1}.
$$
If $c_1 = c_2$, then a similar calculation yields
$y(y^{r-1} + (-1)^{r})/(y+1) = (y^{r} + (-1)^{r}y)/(y+1)$.
\endproof

From this lemma, it is easy to derive well-known chromatic polynomials for cycles.

\begin{lemma}
\label{lem:chr_poly_cycles}
The chromatic polynomial of the $n$-cycle $C_n$ is equal to
$$
   Q_{C_n}(y) = \left\{
                  \begin{array}{ll}
                    y^n - y, & \hbox{$n$ is odd;} \\
                    y^n + y, & \hbox{$n$ is even.}
                  \end{array}
                \right.
$$
\end{lemma}

A \emph{theta graph} is a (simple) graph consisting of two vertices of degree 3 that are joined by three internally disjoint paths.

\begin{lemma}
\label{lem:theta}
Let $G$ be a theta graph on $n$ vertices.
Then for $y \geq 2$, the number of $(y+1)$-colourings of $G$ is at most $y^n - y^{n-1} + y^{n-2} + 2y^{n-3} - y^{n-4}$, with equality if and only if $G = K_{2, 3}$.
In particular, $Q_G(y) < y^n - \frac{4}{9} y^{n-1}$ for $y \in [3, \infty)$.
\end{lemma}

\proof Observe that
\begin{equation}
  (y+1)(y^n - y^{n-1} + y^{n-2} + 2y^{n-3} - y^{n-4}) = y^{n+1} + 3y^{n - 2} + y^{n - 3} - y^{n - 4}.
  \label{eq:Lemma2.3}
\end{equation}
Let $v_1$ and $v_2$ be the vertices of $G$ of degree $3$ and let $r_1 \leq r_2 \leq r_3$ be the lengths of the paths between them,
so that $r_1 + r_2 + r_3 = n+1$.
By considering separately all colourings where $v_1$ and $v_2$ are coloured differently and the same, respectively, and applying Lemma~\ref{lem:paths}, we can express the (shifted) chromatic polynomial $Q_G(y)$ as
\begin{align*}
&(y+1)y\cdot\frac{y^{r_1} + (-1)^{r_1 + 1}}{y+1} \cdot \frac{y^{r_2} + (-1)^{r_2 + 1}}{y+1} \cdot \frac{y^{r_3} + (-1)^{r_3 + 1}}{y+1} \\
&\qquad+ (y+1)\cdot\frac{y^{r_1} + (-1)^{r_1}y}{y+1} \cdot \frac{y^{r_2} + (-1)^{r_2}y}{y+1} \cdot \frac{y^{r_3} + (-1)^{r_3}y}{y+1} \\
&= \frac{y}{(y+1)^2} \cdot \Bigl( y^{r_1 + r_2 + r_3} + (-1)^{r_1 + 1} y^{r_2 + r_3} + (-1)^{r_2 + 1} y^{r_1 + r_3} + (-1)^{r_3 + 1} y^{r_1 + r_2} \\
&\qquad\qquad\qquad+ (-1)^{r_1 + r_2} y^{r_3} + (-1)^{r_1 + r_3} y^{r_2} + (-1)^{r_2 + r_3} y^{r_1} + (-1)^{r_1 + r_2 + r_3 + 1} \\
&\qquad\qquad\qquad+ y^{r_1 + r_2 + r_3 - 1} + (-1)^{r_1} y^{r_2 + r_3} + (-1)^{r_2} y^{r_1 + r_3} + (-1)^{r_3} y^{r_1 + r_2} \\
&\qquad\qquad\qquad+ (-1)^{r_1 + r_2} y^{r_3+1} + (-1)^{r_1 + r_3} y^{r_2+1} + (-1)^{r_2 + r_3} y^{r_1+1} + (-1)^{r_1 + r_2 + r_3}y^2 \Bigr) \\[2mm]
&= \frac{y^{r_1 + r_2 + r_3} + (-1)^{r_1 + r_2} y^{r_3+1} + (-1)^{r_1 + r_3} y^{r_2+1} + (-1)^{r_2 + r_3} y^{r_1+1} + (-1)^{r_1 + r_2 + r_3}y(y-1)}{y+1}.
\end{align*}

If $r_1 \geq 2$ then the inequality is immediate from (\ref{eq:Lemma2.3}),
with equality if and only if $r_2 = r_3 = 2$.
So we may assume that $r_1 = 1$.
The numerator in the last line above becomes
$y^{r_2 + r_3 + 1} + (-1)^{r_2 + 1}y^{r_3 + 1} + (-1)^{r_3 + 1}y^{r_2 + 1} + (-1)^{r_2 + r_3}y$.
If one of the terms has a negative coefficient, we can leave it out in obtaining an upper bound and the claims follow easily. Thus, we may assume that all coefficients are positive.
In order for the second term to be positive, $r_2$ must be odd.
Since $r_2 > 1$ (or we would have a double edge between $x_1$ and $x_2$) we have $r_2 \geq 3$ and thus also $r_3 \geq 3$.
Now again the (strict) inequality is immediate.
\endproof

\begin{lemma}
\label{lem:delta2}
Suppose that $G$ is a graph with $\delta(G)\ge2$ that is not a cycle. Then $G$ contains a subgraph $F$ consisting of two cycles $C_1$ and $C_2$ that are either disjoint or their intersection is a path (possibly a single vertex). Moreover, if $G$ is not bipartite, then we may take $C_1$ to be shortest odd cycle in $G$.

Consequently, if $G$ is connected and $n=|G|$, then for every $y\ge3$,
\begin{equation}\label{eq:delta2-bip}
  Q_G(y) \le y^{n-4} (y^4 - y^3 + y^2 + 2y - 1),
\end{equation}
where equality holds if and only if $G$ is isomorphic to $K_{2,3}$. If $G$ contains a triangle, then
\begin{equation}\label{eq:delta2-triangle}
  Q_G(y) \le y^{n-4} (y^4 - y^3 + y - 1).
\end{equation}
\end{lemma}

\begin{proof}
The proof of the first claim is trivial when $G$ is disconnected. If $G$ has a cutedge $e$, then we consider the components of $G-e$. Each of them has at most one vertex of degree 1, so each of them contains a cycle; and if $G$ is nonbipartite, one of them is nonbipartite and contains an odd cycle. Thus, we may assume that $G$ is 2-edge-connected. Take a shortest cycle $C_1$ (shortest odd cycle if $G$ is nonbipartite) in $G$. This cycle is induced in $G$ and since $G$ is not a cycle, it has a vertex $v\notin V(C_1)$ that is adjacent to some vertex $u\in V(C_1)$. By adding the edge $uv$ and a shortest path in $G-uv$ from $v$ to $C_1$, we obtain a desired subgraph $F$.

To prove the second part, let us denote the RHS in (\ref{eq:delta2-bip}) as $g(y)$. If $F$ is the theta graph, the claim follows from Lemma \ref{lem:theta}. Next, assume that $C_1$ and $C_2$ are disjoint or that $C_1\cap C_2$ is a single vertex. By using Lemmas \ref{lem:decomposition} and \ref{lem:chr_poly_cycles}, we conclude that
$$
   Q_G(y) \le y^{n-|C_1|-|C_2|+1} (y^{|C_1|} + (-1)^{|C_1|}y) (y^{|C_2|} + (-1)^{|C_2|}y)/(y+1).
$$
If $|C_1|,|C_2|\ge 4$, then
\begin{eqnarray*}
   Q_G(y) &\le& y^{n-|C_1|-|C_2|+1} (y^{|C_1|} + y) (y^{|C_2|} + y)/(y+1)\\
   &\le& y^{n-7} (y^4 + y) (y^4 + y)/(y+1)\\
   & = & y^{n-5} (y^3 + 1) (y^2 - y + 1)\\
   & = & y^n - y^{n-1} + y^{n-2} + y^{n-3} - y^{n-4} + y^{n-5} < g(n).
\end{eqnarray*}
Similarly, if $|C_1|=3$, then we have the worst outcome if $|C_2|=4$ and we have
\begin{eqnarray*}
   Q_G(y) &\le& y^{n-6} (y^3 - y) (y^4 + y)/(y+1)\\
   & = & y^{n-4} (y-1) (y^3 + 1)\\
   & = & y^n - y^{n-1} + y^{n-3} - y^{n-4}.
\end{eqnarray*}
This proves (\ref{eq:delta2-triangle}) and completes the proof.
\end{proof}

\section{List colouring}
\label{sect:list chromatic polynomials}

Given a graph $G$, a \emph{list assignment} for $G$ is a function $L : V(G) \rarr \mP(\N)$.
If $L$ is a list assignment for $G$, an \emph{$L$-colouring} of $G$ (or a \emph{list-colouring} using list assignment $L$) is a proper colouring $c$ of $G$ such that $c(v) \in L(v)$ for every $v \in V(G)$.
We write $\norm{L}$ for the number of $L$-colourings of $G$.
We say that $G$ is \emph{$k$-choosable} if $\norm{L}>0$ for every list assignment for which $|L(v)|\ge k$ for every $v \in V(G)$.

Let $G$ be a bipartite graph on vertex classes $V^+$ and $V^-$.
Let $L$ be a list assignment for $G$ and let $i, j$ be distinct positive integers.
For $S \subseteq \N$, we define the \emph{$ij$-compression} $C_{ij}(S)$ by
$$
C_{ij}(S) = \begin{cases}
(S \cup \{i\}) \sm \{j\}, &\text{ if } j \in S \text{ and } i \notin S; \\
S &\text{ otherwise.}
\end{cases}
$$
Then we define the $ij$-compression $C_{ij}(L)$ of $L$ as follows:
$$
(C_{ij}(L))(v) = \begin{cases}
C_{ij}(L(v)), &\text{ if } v \in V^+; \\
C_{ji}(L(v)), &\text{ if } v \in V^-. \\
\end{cases}
$$

For an $L$-colouring $c$ of $V(G)$, an \emph{$ij$-Kempe component} of $G$ is a component of the induced subgraph of $G$ on the union of the colour classes $i$ and $j$ under $c$.

\begin{lemma}
\label{lem:compression}
Let $G$ be a bipartite graph on vertex classes $V^+$ and $V^-$.
Let $L: V^+ \cup V^- \rarr \mP(\N)$ be a list assignment for $G$ and let $i, j$ be distinct positive integers.
Then $\norm{C_{ij}(L)} \geq \norm{L}$.
\end{lemma}

\proof We define an injective function $\phi$ which takes an $L$-colouring of $G$ to a $C_{ij}(L)$-colouring of $G$.
Let $S$ be the set of vertices $v \in V(G)$ such that $C_{ij}(L)(v) \neq L(v)$.
For any colouring $c$ of $G$, we derive $\phi(c)$ from $c$ by swapping the colours $i$ and $j$ on every $ij$-Kempe component of $G$ which intersects $S$.
Observe that $\phi$ is an involution; hence, $\phi$ is injective.

It suffices to show that if $c$ is an $L$-colouring of $G$, then $c' := \phi(c)$ is an $L'$-colouring of $G$, where $L' := C_{ij}(L)$.
Indeed, suppose for a contradiction that there exists $v \in V(G)$ such that $c'(v) \notin L'(v)$.
Since $c(v) \in L(v)$ we have $c(v), c'(v) \in \{i, j\}$.
If $c(v) = c'(v)$ then we must have $L(v) \neq L'(v)$ and hence $v \in S$, but in this case the colour of $v$ should have been changed;
this gives a contradiction.
So $c(v) \neq c'(v)$.

Let $H$ be the $ij$-Kempe component of $G$ containing $v$.
We claim that $c(z) = j$ for every $z \in V(H) \cap V^+$ and $c(z) = i$ for every $z \in V(H) \cap V^-$.
Indeed, $V(H)$ intersects $S$, since the colour of $v$ was swapped.
Let $w \in V(H) \cap S$ and note that $c(w) \in \{i, j\}$.
Suppose that $w \in V^+$ (the case when $w \in V^-$ is similar).
Then $j \in L(w)$ and $i \notin L(w)$, since $w \in S$.
Since $c(w) \in L(w)$ we have that $c(w) = j$.
The claim now follows from the definition of $H$.

Now we may assume without loss of generality that $v \in V^+$.
Then $c(v) = j$ and so $c'(v) = i$.
If $i \in L(v)$ then $i \in C_{ij}(L(v)) = L'(v)$, contradicting our assumption that $c'(v) \notin L'(v)$.
So $i \notin L(v)$.
But then $L'(v) = C_{ij}(L(v)) = (L(v) \cup \{i\}) \sm \{j\}$ and again we have a contradiction.
This proves the lemma.
\endproof

\begin{cor}
\label{cor:list colourings}
Let $G$ be a bipartite graph on vertex classes $V^+$ and $V^-$, let $m \in \N$ and let $\kappa : V(G) \rarr [m]$.
Then among all list assignments $L$ for $G$ such that $|L(v)| = \kappa(v)$ and $\max L(v) \leq m$ for every $v \in V(G)$, $\norm{L}$ is maximized by the list assignment $L_0$ which assigns to every $v \in V^+$ an initial segment of $[m]$ and to every $v \in V^-$ a terminal segment of $[m]$ (of length $\kappa(v)$ in each case).
\end{cor}

\proof
Let $L$ be any list assignment for $G$ such that $|L(v)| = \kappa(v)$ and $\max L(v) \leq m$ for every $v \in V(G)$ and such that $\norm{L}$ is maximum under this condition.
By repeatedly applying $ij$-compressions for $i < j$ we can transform $L$ into $L_0$.
By Lemma~\ref{lem:compression} these compressions do not decrease $\norm{L}$,
and the corollary follows immediately.
\endproof

A special case of the above setup where $\kappa(v)=y$ for each vertex $v$ was studied by Brown, Erey, and Li in \cite{BEL16}. Corollary \ref{cor:list colourings} in particular solves Conjecture 3.1 from their work.

In the setting of chromatic polynomials it makes more sense to specify the \emph{forbidden colours} for each vertex. Suppose that we consider colourings with $y+1$ colours and that $L'(v)\subseteq [y+1]$ are the colours that cannot be used to colour the vertex $v$. This is the list-colouring problem with lists being defined as $L(v) = [y+1]\sm L'(v)$.

Let $L'$ be a list-assignment of forbidden colours for a graph $G$. Let $C' = \cup_{v\in V(G)} L'(v)$ and suppose that $C'\subset [y+1]$. Then we define $Q_{G,L'}(y)$ as the number of $(y+1)$-colourings of $G$ in which no forbidden colour is used at any vertex. If $e=uv\in E(G)$, then we define $L''(e)=L'(u)\cup L'(v)$ and we set $L''(w) = L'(w)$ for all vertices $w\in V(G)\sm\{u,v\}$. The following deletion-contraction formula is easy to prove:
$$
   Q_{G,L'}(y) = Q_{G-e,L'}(y) - Q_{G/e,L''}(y).
$$
As a consequence we have the following observation.

\begin{proposition}
$Q_{G,L'}(y)$ is a monic polynomial in $y$ of degree $|G|$ for all values of $y$ such that $C' = \cup_{v\in V(G)} L'(v) \subseteq [y+1]$.
\end{proposition}

The polynomial $Q_{G,L'}(y)$ is called the \emph{list chromatic polynomial\/} of $G$ with respect to the forbidden list assignment $L'$.
We will need upper bounds on the values of some list chromatic polynomials. Let us start with the case when $G$ is a path $P_n$ on $n$ vertices $\{v_1,v_2,\dots,v_n\}$ (and edges $v_iv_{i+1}$, $1\le i<n$). Suppose first that each vertex has precisely one forbidden colour. By Corollary \ref{cor:list colourings}, the largest possible number of list colourings is attained when $L'(v_i)=\{1\}$ if $i$ is odd and $L'(v_i)=\{4\}$ if $i$ is even.\footnote{Here we write 4 instead of $y+1$ as used in the corollary. Of course, this is irrelevant.}
Let $A_n(y) = Q_{P_n,L'}(y)$. We will also consider the cases when $L'(v_1)=\{1,2\}$ (and the rest as before). In this case, we denote the corresponding polynomial by $B_n(y) = Q_{P_n,L'}(y)$. If we also have two forbidden colours for the other end of the path ($L'(v_n)=\{1,2\}$ or $\{3,4\}$, depending on the parity of $n$), then we denote the corresponding chromatic polynomial by $C_n(y)$. It is easy to see that the following recurrence holds:
\begin{align}
  A_n(y) & = A_{n-1}(y) + (y-1)B_{n-1}(y) \nonumber \\
  B_n(y) & = A_{n-1}(y) + (y-2)B_{n-1}(y) \nonumber \\
  C_n(y) & \le B_{n-1}(y) + (y-2)C_{n-1}(y). \label{eq:recurrence paths}
\end{align}
The last recurrence starts with $n=3$. The initial conditions are: $A_1(y) = y$, $B_1(y)=y-1$ and $C_2(y)=y^2-3y+4$.
For our purpose it will suffice to have an upper bound $\widehat{C}_n(y)$ on  $C_n(y)$, which is obtained by solving the last recurrence in (\ref{eq:recurrence paths}) when the inequality is replaced by equality.
${A}_n(y)$ and ${B}_n(y)$ are easy to compute for small values of $n$, and they are collected in Table \ref{tab:1} for our further use. As for the third kind, $\widehat{C}_n(y)$, we will only need the following one:
\begin{equation}\label{eq:C_5(y)}
  \widehat{C}_5(y) = y^5-6y^4+19y^3-34y^2+33y-13.
\end{equation}

Of course one can also easily solve the linear recurrence for these polynomials.

\begin{table}[htb]
\label{tab:1}
\begin{center}
$
\begin{array}{|c||l|l|}
  \hline
  n & {A}_n(y) & {B}_n(y) \\
  \hline
  1 & y & y-1 \\
  2 & y^2-y+1 & y^2-2y+2 \\
  3 & y^3-2y^2+3y-1 & y^3-3y^2+5y-3 \\
  4 & y^4-3y^3+6y^2-5y+2 & y^4-4y^3+9y^2-10y+5 \\
  5 & y^5-4y^4+10y^3-13y^2+10y-3 & y^5-5y^4+14y^3-22y^2+20y-8 \\
  6 & y^6-5y^5+15y^4-26y^3+29y^2-18y+5 & y^6-6y^5+20y^4-40y^3+51y^2-38y+13 \\
  \hline
\end{array}
$
\end{center}
\caption{Values of ${A}_n(y)$ and ${B}_n(y)$, the list chromatic polynomials of paths.}
\end{table}

We will need some other list chromatic polynomials. In the proof of the next result and also later on, we will use the following \emph{truth function} $\tau$.
Given a proposition $A$, we set $\tau(A)=1$ if $A$ is true and $\tau(A)=0$ if $A$ is false.

\begin{proposition}
\label{prop:C3,C5,K23}
Let\/ $G$ be a graph and let $y\ge3$ be an integer. Suppose that the list $L'(v)\subseteq [y+1]$ of forbidden colours is nonempty for each vertex of $G$. Then we have the following:
\begin{itemize}
  \item[\rm (a)] If $G=K_3$, then
  $Q_{K_3,L'}(y) \le y^3 -3y^2 + 5y -4$ with equality for $y\ge3$ if and only if $|L'(v)|=1$ for each vertex of $K_3$ and the forbidden colours are distinct for all three vertices.
  \item[\rm (b)] If $G=K_{2,3}$, then $Q_{K_{2,3},L'}(y) \le y^5 - 6y^4 + 21y^3 -38y^2 + 36y - 13$ with equality if and only if $|L'(v)|=1$ for each vertex of $K_{2,3}$ and the forbidden colours are the same in each bipartite class and different between the classes.
  \item[\rm (c)] If $G=C_5$, then $Q_{C_5,L'}(y) \le y^5 - 5y^4 + 15y^3 - 28y^2 + 31y - 16$, with equality if and only if $c_i \neq c_{i+1}$ for every $i \in [5]$ (where $c_6=c_1$) and $c_i = c_{i+2}$ for two indices $i \in [5]$.
\end{itemize}
\end{proposition}

\begin{proof}
We could proceed in a similar way as when finding the recurrence for the paths. Note, however, that $C_3$ and $C_5$ are not bipartite.
We will use a different counting method based on inclusion-exclusion principle.
We will denote the vertices of $G$ by $v_1,\dots,v_n$, $n=|G|$, and for each $i\in [n]$, pick a forbidden colour $c_i\in L'(v_i)$.
For each $I \subseteq [n]$, let $\Sigma_I$ be the number of colourings of $G$ such that $c(x_i) = c_i$ for each $i \in I$. Knowing these numbers, we can determine the number of colourings for which $c(v_i)\ne c_i$ ($i=1,\dots,n$) by the inclusion-exclusion principle. In other words, we have:
\begin{equation}
    Q_{G,L'}(y) \le \sum_{I\subseteq [n]} (-1)^{|I|} \Sigma_I, \label{eq:inclusion-exclusion}
\end{equation}
with equality in (\ref{eq:inclusion-exclusion}) if $|L(v_i)|=1$ for every $i\in [n]$.

(a) $G=K_3$. In this case, $\Sigma_\emptyset = y(y^2-1)$, $\Sigma_{\{i\}} = y(y-1)$ ($i\in [3]$), $\Sigma_{\{i,j\}} = \tau(c_i\ne c_j)(y-1)$ ($1\le i<j\le 3$), and $\Sigma_{\{1,2,3\}} = \tau(c_1\ne c_2) \tau(c_1\ne c_3) \tau(c_2\ne c_3)$. If $c_1,c_2,c_3$ are all different, then
\begin{align}
Q_{K_3,L'}(y)
&\leq \Sigma_{\emptyset} - \Sigma_{\{1\}} - \Sigma_{\{2\}} - \Sigma_{\{3\}} + \Sigma_{\{1, 2\}} + \Sigma_{\{1, 3\}} +\Sigma_{\{2, 3\}} - \Sigma_{\{1, 2, 3\}} \nonumber\\
&=y(y^2-1) - 3y(y-1) + 3(y-1) -1 = y^3 -3y^2 + 5y - 4.
\nonumber
\end{align}
It is easy to see that equality holds if and only if there are no other forbidden colours.

Suppose now that $c_1=c_2$. Then we obtain in the same way:
$$
Q_{K_3,L'} \leq y(y^2-1) - 3y(y-1) + 2(y-1) = y^3-3y^2+4y-2 < y^3 -3y^2 + 5y -4.
$$

For the proof of (b), let $v_1, v_2$ be the vertices of $K_{2,3}$ of degree $3$ and let $v_3, v_4, v_5$ be the vertices of degree $2$.
By Corollary~\ref{cor:list colourings} we may assume that $c_1 = c_2 = 1$ and $c_3 = c_4 = c_5 = y+1$.
Observe that
\begin{align*}
\Sigma_{\emptyset} &= y^5 - y^4 + y^3 + 2y^2 - y \text{ (see Lemma \ref{lem:theta}),}\\
\Sigma_{\{i\}} &= y(y-1)^3 + y^3 \text{ for } i \in [5], \\
\Sigma_{\{1, 2\}} &= y^3, \\
\Sigma_{\{i, j\}} &= y^2 + (y-1)^3 \text{ for } i \in \{1,2\} \text{ and } j \in \{3,4,5\}, \\
\Sigma_{\{j, k\}} &= y(y^2 - y + 1) \text{ for distinct } j, k \in \{3, 4, 5\}, \\
\Sigma_{\{1, 2, j\}} &= y^2 \text{ for } j \in \{3,4,5\}, \\
\Sigma_{\{i, j, k\}} &= y^2 - y + 1 \text{ for } i \in \{1,2\} \text{ and distinct } j, k \in \{3, 4, 5\}, \\
\Sigma_{\{3, 4, 5\}} &= y^2, \\
\Sigma_{\{1, 2, j, k\}} &= y \text{ for distinct } j, k \in \{3, 4, 5\}, \\
\Sigma_{\{i, 3, 4, 5\}} &= y \text{ for } i \in \{1,2\}, \text{ and }\ \Sigma_{\{1, 2, 3, 4, 5\}} = 1.
\end{align*}

By the inclusion-exclusion principle, the number of list colourings is
\begin{align*}
\sum_{I \subseteq [5]} (-1)^{|I|} \Sigma_I &= (y^5 - y^4 + y^3 + 2y^2 - y) -5(y(y-1)^3 + y^3) +y^3 +6(y^2 + (y-1)^3) \\
&\qquad+3y(y^2 - y + 1) -3y^2 -6(y^2 - y + 1) -y^2 +5y -1 \\
&= y^5 - y^4 -5y^2 + 13y - 7 - 5y(y^3 - 3y^2 + 3y - 1) + 6(y^3 - 3y^2 + 3y - 1) \\
&= y^5 - 6y^4 + 21y^3 -38y^2 + 36y - 13.
\end{align*}

If there are additional forbidden colours or the forbidden colours are different than the case treated above, the number of list colourings is strictly smaller. This confirms the statement about when equality holds.

(c) We order the vertices of $C_5$ naturally as $v_1, \ldots, v_5$.
We adopt the convention that $c_{5 + j} = c_j$.

Suppose first that $c_i \neq c_{i+1}$ for every $i \in [5]$. Then we have:
\begin{align*}
\Sigma_{\emptyset} &= y^5 - y \text{ (see Lemma \ref{lem:chr_poly_cycles}),}\\
\Sigma_{\{i\}} &= y^4 - y^3 + y^2 - y \text{ for } i \in [5], \\
\Sigma_{\{i, i+1\}} &= y^3 - y^2 + y - 1 \text{ for } i \in [5], \\
\Sigma_{\{i, i+2\}} &= y^3 - 2y^2 + 2y - 1 + \tau(c_{i} = c_{i+2})(y^2 - 2y + 1) \text{ for } i \in [5], \\
\Sigma_{\{i, i+1, i+2\}} &= y^2 - y + 1 - \tau(c_{i} = c_{i+2}) \text{ for } i \in [5], \\
\Sigma_{\{i, i+1, i+3\}} &= (y-1)^2 + \tau(c_{i} = c_{i+3})(y-1) + \tau(c_{i+1} = c_{i+3})(y-1) \text{ for } i \in [5], \\
\Sigma_{\{i, i+1, i+2, i+3\}} &= y - 1 + \tau(c_{i} = c_{i+3}) \text{ for } i \in [5], \\
\Sigma_{[5]} &= 1.
\end{align*}
We observe that the coefficient of $\tau(c_{i} = c_{i+2})$ in $\sum_{I \in \mI} (-1)^{|I|} \Sigma_I$ is $y^2 - 4y + 5 > 0$.
Hence we may assume that the number of indices $i$ such that $c_i = c_{i+2}$ is as large as possible.
By the pigeonhole principle there are at most two such pairs.
Now we obtain
\begin{align*}
\sum_{I \in \mI} (-1)^{|I|} \Sigma_I
&\le y^5 - y - 5(y^4 - y^3 + y^2 - y) \\
&\qquad+ 5(y^3 - y^2 + y - 1) + 5(y^3 - 2y^2 + 2y - 1) \\
&\qquad- 5(y^2 - y + 1) - 5(y^2 - 2y + 1)
+ 5(y-1) - 1 + 2(y^2-4y+5) \\
&= y^5 - 5y^4 + 15y^3 - 28y^2 + 31y - 16.
\end{align*}

Now suppose that $c_i = c_{i+1}$ for some $i \in [5]$.
Without loss of generality we may assume that $c_1 = c_5$.
By similar arguments to those in the proof of Lemma~\ref{lem:compression} we may assume that $c_1 \neq c_2$, $c_2 \neq c_3$, $c_3 \neq c_4$, $c_4 \neq c_5$, and $c_2 = c_4$ and $c_1=c_3$.
Now, $\Sigma_I = 0$ whenever $\{1,5\}\subseteq I$ and
\begin{align*}
\Sigma_{\emptyset} &= y^5 - y \text{ (see Lemma \ref{lem:chr_poly_cycles}),}\\
\Sigma_{\{i\}} &= y^4 - y^3 + y^2 - y \text{ for } i \in [5], \\
\Sigma_{\{i, i+1\}} &= y^3 - y^2 + y - 1 \text{ for } i \in [4], \\
\Sigma_{\{i, i+2\}} &= y^3 - y^2 \text{ for } i \in [3], \\
\Sigma_{\{i, i+2\}} &= y^3 - 2y^2 + 2y - 1 \text{ for } i \in \{4, 5\}, \\
\Sigma_{\{i, i+1, i+2\}} &= y^2 - y \text{ for } i \in [3], \\
\Sigma_{\{i, i+1, i+3\}} &= y^2 - y \text{ for } i \in [4], \\
\Sigma_{\{i, i+1, i+2, i+3\}} &= y - 1 \text{ for } i \in [2], \\
\Sigma_{[5]} &= 1.
\end{align*}

Hence
\begin{align*}
\sum_{I \in \mI} (-1)^{|I|} \Sigma_I &= y^5 - y - 5(y^4 - y^3 + y^2 - y) \\
&\qquad+ 4(y^3 - y^2 + y - 1) + 3(y^3 - y^2) + 2(y^3 - 2y^2 + 2y - 1) \\
&\qquad- 7(y^2 - y) + 2(y-1) - 1 \\[2mm]
&= y^5 - 5y^4 + 14y^3 - 23y^2 + 21y - 9 \\
&\leq y^5 - 5y^4 + 15y^3 - 28y^2 + 31y - 16.
\end{align*}
\end{proof}

\begin{proposition}
\label{prop:C4 plus leaf}
Let\/ $G$ be a graph obtained from a $4$-cycle $v_1v_2v_3v_4v_1$ by adding a vertex $v_5$ adjacent to $v_1$ and let $y\ge3$ be an integer. Suppose that the lists $L'(v_i)\subseteq [y+1]$ of forbidden colours are nonempty for $i=2,3,4,5$. Then we have the following:
\begin{itemize}
  \item[\rm (a)]
  $Q_{G,L'}(y) \le y^5 - 4y^4 + 10y^3 - 13y^2 + 10y - 3$.
\smallskip
  \item[\rm (b)]
  If\/ $|L'(v_5)|\ge2$, then $Q_{G,L'}(y) \le y^5 - 5y^4 + 14y^3 - 23y^2 + 23y - 11$.
\end{itemize}
\end{proposition}

\begin{proof}
(a) Let $c_i\in L'(v_i)$ be an arbitrarily chosen forbidden colour for $i \in \{2, 3, 4, 5\}$.
By Corollary~\ref{cor:list colourings} we may assume that $c_2 = c_4 = c_5 = 1$ and that $c_3 = y+1$.
As before, for each $I \subseteq \{2, 3, 4, 5\}$, let $\Sigma_I$ be the number of proper $(y+1)$-colourings of $G$ such that $c(v_i) = c_i$ for each $i \in I$.
By the inclusion-exclusion principle, $Q_{G,L'}(y) \le \sum_{I \subseteq \{2, 3, 4, 5\}} (-1)^{|I|} \Sigma_I$.
Observe that $\Sigma_{\emptyset} = y^5 + y^2$ (see Lemma \ref{lem:chr_poly_cycles}),
$\Sigma_{\{i\}} = y^4 - y^3 + y^2$ for $i \in \{2, 3, 4, 5\}$,
$\Sigma_{\{2, 4\}} = y^3$,
$\Sigma_{\{i, j\}} = y^3 - y^2 + y$ for $i \in \{2, 4\}$ and $j \in \{3, 5\}$,
$\Sigma_{\{3, 5\}} = (y-1)^3 + y^2$,
$\Sigma_{\{2, 3, 4\}} = \Sigma_{\{2, 4, 5\}} = y^2$,
$\Sigma_{\{i, 3, 5\}} = y^2 - y + 1$ for $i\in\{2, 4\}$, and
$\Sigma_{\{2, 3, 4, 5\}} = y$.
Hence
\begin{align*}
\sum_{I \subseteq \{2, 3, 4, 5\}} (-1)^{|I|} \Sigma_I &= y^5 + y^2 - 4(y^4 - y^3 + y^2) + y^3 + 4(y^3 - y^2 + y) \\
&\qquad+ (y-1)^3 + y^2 - 2y^2 - 2(y^2 - y + 1) + y \\[2mm]
&= y^5 - 4y^4 + 10y^3 - 13y^2 + 10y - 3.
\end{align*}

(b) Suppose that $L'(v_5)$ contains distinct colours $c_5$ and $c_6$.
By Corollary~\ref{cor:list colourings} we may assume that $c_2 = c_4 = c_5 = 1$, $c_6 = 2$ and $c_3 = y+1$.
Let $\mI$ be the family of subsets of $\{2, 3, 4, 5, 6\}$ which do not contain $\{5, 6\}$ as a subset.
For each $I \in \mI$ we define $\Sigma_I$ as before, except that if $6 \in I$ then we have $c(v_5) = c_6$.
Then $Q_{G,L'}(y) \le \sum_{I \in \mI} (-1)^{|I|} \Sigma_I$.
We have already calculated $\Sigma_I$ for $I \in \mI$ such that $6 \notin I$.
Observe that
\begin{align*}
&\Sigma_{\{6\}} = y^4 - y^3 + y^2, &
&\Sigma_{\{i, 6\}} = (y-1)(y^2 - y + 1) \text{ for } i = 2, 4, \\
&\Sigma_{\{3, 6\}} = (y-1)^3 + y^2, &
&\Sigma_{\{i, 3, 6\}} = y^2 - 2y + 2 \text{ for } i = 2, 4, \\
&\Sigma_{\{2, 4, 6\}} = y(y-1), \text{ and } &
&\Sigma_{\{2, 3, 4, 6\}} = y - 1.
\end{align*}

\noindent
Hence
\begin{align*}
\sum_{I \in \mI} (-1)^{|I|} \Sigma_I &= (y^5 - 4y^4 + 10y^3 - 13y^2 + 10y - 3) - (y^4 - y^3 + y^2) + 2(y^3 - 2y^2 + 2y - 1) \\
&\qquad+ (y-1)^3 + y^2 - 2(y^2 - 2y + 2) - y(y-1) + y - 1 \\[2mm]
&= y^5 - 5y^4 + 14y^3 - 23y^2 + 23y - 11.
\end{align*}
\end{proof}

\section{Proof of Theorem~\ref{thm:main}}

Fix $x \geq 4$ and let $y=x-1\ge3$.
The proof is by induction, and we consider a minimal counterexample $G$.
If $G$ has two vertices that are coloured the same in every $(y+1)$-colouring, then we identify the two vertices and apply induction. Thus, we may assume that $G$ and any subgraph of $G$ has no such vertices. Moreover,
we may assume that $G$ is vertex-critical for chromatic number 4.
Indeed, if $G$ is not vertex-critical then let $F$ be a minimal $4$-vertex-critical subgraph of $G$.
Then $|F| < |G|$, and by induction $Q_F(y) \leq (y+1)y^{|F| - 3}(y-1)(y-2)$.
Further, since $\delta(F) \geq 2$, equality holds if and only if $F$ is isomorphic to $K_4$.
Now Lemma~\ref{lem:subgraph} implies that $Q_G(y) \leq y^{|G| - |F|}Q_F(y) \leq (y+1)y^{|G| - 3}(y-1)(y-2)$.
Equality holds if and only if $F \simeq K_4$ and $G$ is formed from $F$ by appending vertex-disjoint trees to the vertices of $F$, as desired.

Since $G$ is vertex-critical it has minimum degree at least $3$ and is $2$-connected.
From the latter we may assume that $G$ is edge-critical.
Indeed, deleting an edge of $G$ leaves a connected graph $G'$ with at least as many $(y+1)$-colourings as $G$.
If $\chi(G') = 4$, then we apply the induction hypothesis and conclude that (\ref{eq:3}) holds since $G'$ is 4-critical but different from $K_4$. Thus, $G$ also fulfills (\ref{eq:3}) and thus is not a counterexample.

\begin{claim}\label{claim:twoInduced}
Let $F$ be a non-empty $2$-induced proper subgraph of $G$.
Then there are fewer than $y^{n - |F|}$ ways to $(y+1)$-colour $G - V(F)$.
\end{claim}

To prove the claim, we first observe that no component $C$ of $G - V(F)$ can be $2$-choosable
-- otherwise we could extend a $3$-colouring of $G - V(C)$ to $C$ -- and that $\delta(C) \geq 2$.
In particular, $C$ cannot be an even cycle.
If $C$ is an odd cycle, Lemma \ref{lem:chr_poly_cycles} (combined with Lemma \ref{lem:subgraph}) completes the proof of the claim; otherwise, Lemma \ref{lem:delta2} does the same.

\medskip

The rest of the proof of Theorem \ref{thm:main} splits into several cases.
Each case implictly excludes all earlier cases.

\medskip

\Case{1}{$G$ has $7$ or fewer vertices.}

In this case there are only four graphs to consider, and one of them is $K_4$.
We just compute their chromatic polynomials to confirm the theorem. The details are left to the reader.

\medskip

\Case{2}{$G$ contains $K_4^-$ as a subgraph.}

Let $F$ be a copy of $K_4^-$ in $G$.
We may assume that $F$ is an induced subgraph of $G$.
Let $x_1$ and $x_2$ be the vertices of degree $2$ in $F$ and let $x_3$ and $x_4$ be the vertices of degree $3$ in $F$.
We form a graph $G'$ from $G$ by identifying $x_1$ and $x_2$.
Observe that $\chi(G') \geq 4$ and that $|E(G')| \geq |G'| + 2$.

Suppose first that $F$ is not $2$-induced in $G$;
let $v \in V(G) \sm V(F)$ have two neigbours in $V(F)$.
We first count colourings in which $x_1$ and $x_2$ have the same colour;
this is equivalent to counting colourings of $G'$.
Since $\chi(G') \geq 4$, by induction $G'$ has at most $(y+1)y^{n-4}(y-1)(y-2)$ $k$-colourings.
On the other hand, the number of colourings in which $x_1$ and $x_2$ have different colours is equal to the number of $(y+1)$-colourings of the graph $G^*$
formed from $G$ by adding the edge $x_1$ and $x_2$, creating a copy of $K_4$.
Let $F^* = G^*[V(F) \cup \{v\}]$.
If $v$ has more than $2$ neighbours in $V(F)$ then Lemma~\ref{lem:subgraph} applied with $G = G^*$ and $F = F^*$ implies that $Q_{G^*}(y) \leq y^{n - 5}Q_{F^*}(y) < (y+1)y^{n-4}(y-1)^2(y-2)$.
On the other hand, if $v$ has exactly $2$ neighbours in $V(F)$ then Lemma~\ref{lem:subgraph} applied with $G = G^*$ and $F = F^*$ implies that $Q_{G^*}(y) < y^{n - 5}Q_{F^*}(y) = (y+1)y^{n-4}(y-1)^2(y-2)$;
the inequality is strict since no pair of vertices of $F^*$ has the same colour in every $(y+1)$-colouring of $F^*$. In total we have fewer than
$$(y+1)y^{n-4}(y-1)(y-2) + (y+1)y^{n-4}(y-1)^2(y-2) = (y+1)y^{n-3}(y-1)(y-2)$$
ways to $(y+1)$-colour $G$, as desired. This completes the proof when $F$ is not 2-induced.

Next suppose that neither $x_3$ nor $x_4$ has a neighbour outside $F$.
Let $G'' = G' - \{x_3, x_4\}$.
Observe that $\chi(G'') \geq 4$, as if $G''$ were $3$-colourable we could extend a $3$-colouring of $G''$ to a $3$-colouring of $G'$,
contradicting the fact that $\chi(G') \geq 4$.
By induction $G''$ has at most $(y+1)y^{n-6}(y-1)(y-2)$ $(y+1)$-colourings,
and hence $G'$ has at most $(y+1)y^{n-5}(y-1)^2(y-2)$~ $(y+1)$-colourings.

We next observe that $G - V(F)$ must be connected.
Indeed, suppose for a contradiction that $G - V(F)$ can be partitioned into subgraphs $C_1$ and $C_2$ with no edges between them.
Then we can $3$-colour $G - V(C_1)$ and $G - V(C_2)$ and permute the colour classes so that the colourings agree on $F$,
obtaining a $3$-colouring of $G$, which is a contradiction.
Hence in $G^*$ there is a path $P$ with $q \geq 1$ internal vertices from $x_1$ to $x_2$ which does not pass through $x_3$ or $x_4$.
Now there are at most $(y+1)y^{n-5}(y-1)(y-2)(y^2 - y + 1)$ ways to $(y+1)$-colour $G^*$: $(y+1)y(y-1)(y-2)$ ways to colour $F+x_1x_2$, at most $(y^{q+1} + (-1)^q)/(y+1) \leq y^q - y^{q-1} + y^{q-2}$ ways to extend to the internal vertices of $P$ (Lemma \ref{lem:paths} applied with $r=q+1$ and $c_1\ne c_2$) and at most $y^{n - q - 4}$ ways to extend to the rest of $G^*$.
Thus there are at most
$$(y+1)y^{n-5}(y-1)^2(y-2) + (y+1)y^{n-5}(y-1)(y-2)(y^2 - y + 1) = (y+1)y^{n-3}(y-1)(y-2)$$
ways to $(y+1)$-colour $G$, as desired.

It remains to show that equality cannot occur. Indeed, if we have as many colourings as estimated, then there are precisely $y^q-y^{q-1}+y^{q-2}$ ways to extend any colouring of $F+x_1x_2$ to $P$ and precisely $y^{n-q+4}$ ways to extend to the rest of $G^*$. In the first case, the equality $(y^{q+1}+(-1)^q)/(y+1) = y^q-y^{q-1}+y^{q-2}$ holds if and only if $q=2$.
However, in that case we see, similarly as in the proof of Lemma \ref{lem:subgraph}, that the number of extensions of some colouring of $F+x_1x_2+P$ to $G^*$ is strictly smaller than $y^{n-q+4}$. This completes the proof when $x_3,x_4$ have no neighbours outside $F$.

Finally suppose that $F$ is $2$-induced in $G$ and that $x_3$ has a neighbour outside $F$.
By Claim~\ref{claim:twoInduced} there are fewer than $y^{n-4}$ ways to $(y+1)$-colour $G - V(F)$.
Let $c_i$ be the colour of one of the neighbours of $x_i$ outside $F$ for every $i = 1,2,3$.
We will use the inclusion-exclusion principle to bound the number of ways to extend colourings to $F$.
For each $I \subseteq [3]$, let $\Sigma_I$ be the number of proper colourings of $F$ such that $c(x_i) = c_i$ for each $i \in I$.

Here and in the sequel we will use the following \emph{truth function} $\tau$.
Given a proposition $A$, we set $\tau(A)=1$ if $A$ is true and $\tau(A)=0$ if $A$ is false.
Then the number of ways of extending to $F$ is
\begin{align}
&\Sigma_{\emptyset} - \Sigma_{\{1\}} - \Sigma_{\{2\}} - \Sigma_{\{3\}} + \Sigma_{\{1, 2\}} + \Sigma_{\{1, 3\}} +\Sigma_{\{2, 3\}} - \Sigma_{\{1, 2, 3\}} \nonumber\\
&=(y+1)y(y-1)^2 - y(y-1)^2 - y(y-1)^2 - y(y-1)^2 \nonumber\\
&\qquad+ \tau(c_1 = c_2) y(y-1) + \tau(c_1 \neq c_2) (y-1)(y-2) +
\tau(c_1 \neq c_3)(y-1)^2 + \tau(c_2 \neq c_3)(y-1)^2 \nonumber\\
&\qquad- \tau(c_1 \neq c_3)\tau(c_2 \neq c_3)(\tau(c_1 = c_2) (y-1) + \tau(c_1 \neq c_2) (y-2)) \nonumber\\
&\le y(y-1)^2 (y-2) + 2(y-1)^2 + \tau(c_1 = c_2) (y-1)^2 + \tau(c_1 \ne c_2) (y-2)^2 \nonumber\\
&\leq y(y-1)^2 (y-2) + 3(y-1)^2 \leq (y+1)y(y-1)(y-2). \label{eq:new4}
\end{align}
Hence in total there are fewer than $(y+1)y^{n-3}(y-1)(y-2)$ ways to $(y+1)$-colour $G$, as desired.

It remains to consider what happens when we have equality.
Let us observe that the first inequality in the last line of (\ref{eq:new4}) is strict if $c_1\ne c_2$. Thus, it suffices to prove that there is a colouring of $G$ such that $c_1\ne c_2$. First of all, we may select neighbours $u_i$ of $x_i$ for $i=1,2$ to be distinct. Otherwise, the neighbourhood of one of $x_1,x_2$ would be a subset of the neighbourhood of the other one---and this is not possible in a critical graph. Now we can 3-colour $G-V(F)$ and then change the colour of $u_2$ to colour 4, thus achieving that $c_1\ne c_2$.

\medskip

\Case{3}{$G$ contains two disjoint copies of $K_3$.}

Let $F_1$ and $F_2$ be two disjoint copies of $K_3$ in $G$.
Observe that $F_1$ and $F_2$ are $2$-induced in $G$; otherwise, $G$ would contain $K_4^-$ as a subgraph. Let $G' = G - V(F_2)$.
Also observe that $\delta(G') \ge2$ (since $\delta(G)\ge3$) and that $G' \ne F_1$ since $|G|\ge 8$. By (\ref{eq:delta2-triangle}), $Q_{G'}(y) \le y^{n-7}(y^4-y^3+y-1)$.

By Proposition \ref{prop:C3,C5,K23}(a), the number of ways of extending any colouring of $G'$ to $F_2$ is at most $y^3 - 3y^2 + 5y - 4$.
Hence $Q_G(y) \leq y^{n-7} (y^4-y^3+y-1)(y^3 - 3y^2 + 5y - 4) < (y+1) y^{n-3} (y-1)(y-2)$ for $y \geq 3$, as desired.
The last inequality holds since when the substitution $y = z + 3$ is made, the difference of the RHS and LHS has only non-negative coefficients.

\medskip

\Case{4}{$G$ contains a copy of $K_3$.}

Let $F_1$ be a copy of $K_3$ in $G$.
We may assume that $V(G) \sm V(F_1)$ contains at least $5$ vertices and has minimum degree at least $2$.
There are $y(y^2 - 1)$ ways to $(y+1)$-colour $F_1$.
We claim that we can colour almost all of the remaining vertices of $G$ in such a way that each vertex has at least one already-coloured neighbour when we colour it
(and hence there are at most $y$ ways to colour it),
so that the vertices which remain at the end form an induced cycle.
Indeed, if the minimum degree of the remaining vertices is at least $3$ then we can colour any uncoloured vertex.
Otherwise, let $V'$ be the set of uncoloured vertices and let $v \in V'$ so that $|N(v) \cap V'| = 2$.
If $G[V' \sm \{v\}]$ is acyclic, then $G[V']$ is unicyclic and since it has minimum degree at least $2$, it is a cycle, as desired.
Otherwise, we colour $v$ and recursively colour any vertex with only one neighbour in $V'$.

After this process, if the cycle $C$ induced on $V'$ has length $5$ then we let $F_2$ be this cycle. If it has length more than 5, then
we proceed to colour all of $V'$ except for a path on $5$ vertices, and let $F_2$ be this path.
Otherwise, the cycle $C$ has length $4$. In this case we uncolour the most recently coloured vertex $u\notin F_1\cup C$ that has a neighbour in $C$. Since $|G|>7$ and $G-V(F_1)$ is connected, $u$ exists. Since $G-u$ is connected, the order of colouring can be changed so that all vertices in $G-(V(C)\cup\{u\})$ are coloured first in the desired way. Now, let $F_2$ be the graph induced on $V(C)\cup\{u\}$.
In this case $F_2$ is either $C_4$ plus a leaf or $K_{2, 3}$.

\medskip

\Case{4a}{$F_2$ is a $4$-cycle $(v_1, v_2, v_3, v_4)$ plus a leaf $v_5$ which is adjacent to $v_1$.}

Suppose that $N(v_5) \sm \{v_1\}$ is monochromatic.
We consider the number of ways of extending the colouring to $F_2$.
By Proposition \ref{prop:C4 plus leaf}(a), each such colouring has at most
$y^5 - 4y^4 + 10y^3 - 13y^2 + 10y - 3$ extensions.

Now suppose that $N(v_5) \sm \{v_1\}$ is not monochromatic.
By part (b) of the same proposition, the number of extensions is at most
$y^5 - 5y^4 + 14y^3 - 23y^2 + 23y - 11$.

Let $G' = G - V(F_2)$, let $S = N(v_5) \sm \{v_1\}$ and let $G'' = G' / S$.
Then the number of $(y+1)$-colourings of $G'$ in which $S$ is not monochromatic is $Q_{G'}(y) - Q_{G''}(y)$.
(If $S$ is not an independent set then we consider $G''$ to have a loop and so $Q_{G''}(y) = 0$.)
Now the number of ways to colour $G$ is at most
\begin{align*}
&(y^5 - 4y^4 + 10y^3 - 13y^2 + 10y - 3)Q_{G''}(y) \\
&\qquad+ (y^5 - 5y^4 + 14y^3 - 23y^2 + 23y - 11)(Q_{G'}(y) - Q_{G''}(y)) \\
&= (y^5 - 5y^4 + 14y^3 - 23y^2 + 23y - 11)Q_{G'}(y) + (y^4 - 4y^3 + 10y^2 - 13y + 8)Q_{G''}(y) \\
&\leq (y^5 - 5y^4 + 14y^3 - 23y^2 + 23y - 11)(y+1)y^{n-7}(y-1) \\
&\qquad+ (y^4 - 4y^3 + 10y^2 - 13y + 8)(y+1)y^{n-8}(y-1) \\
&= (y+1)y^{n-8}(y-1)(y^6 - 5y^5 + 15y^4 - 27y^3 + 33y^2 - 24y + 8) \\
&< (y+1)y^{n-3}(y-1)(y-2),
\end{align*}
where the first inequality follows from Lemma~\ref{lem:subgraph} applied with $F = F_1$ and $G = G'$ and the last inequality holds since when the substitution $y = z + 3$ is made, the difference of the RHS and LHS has only positive coefficients.

\medskip

\Case{4b}{$F_2$ is isomorphic to $K_{2,3}$.}

Let $v_1, v_2$ be the vertices of $F_2$ of degree $3$ and let $v_3, v_4, v_5$ be the vertices of $F_2$ of degree $2$.
Since $G$ is $4$-critical, $v_1$ and $v_2$ must each have a neighbour outside $F_2$. This means that each vertex has at least one forbidden colour and we can apply Proposition \ref{prop:C3,C5,K23}(b) to conclude that the number of extensions is at most
\begin{align*}
y^5 - 6y^4 + 21y^3 -38y^2 + 36y - 13 < y^4 (y-2),
\end{align*}
where the inequality holds since when the substitution $y = z + 3$ is made, the difference of the RHS and LHS has only positive coefficients.
It follows immediately that $Q_G(y) < (y+1)y^{n-3}(y-1)(y-2)$.

\medskip

\Case{4c}{$F_2$ is isomorphic to $P_5$.}

We order the vertices of $F_2$ naturally as $v_1, \ldots, v_5$.
The number of extensions will depend on whether the coloured neighbourhoods of $v_1$ and $v_5$ are monochromatic.
If both neighbourhoods are monochromatic, the number of extensions is at most ${A}_5(y) = y^5 - 4y^4 + 10y^3 - 13y^2 +10y - 3$ (see Table \ref{tab:1}).
If exactly one of $v_1$ and $v_5$ has a monochromatic coloured neighbourhood, number of extensions is at most ${B}_5(y) = y^5 - 5y^4 + 14y^3 - 22y^2 + 20y - 8$ (see Table \ref{tab:1}).
Finally, if neither $v_1$ nor $v_5$ has a monochromatic coloured neighbourhood, the number of extensions is at most $\widehat{C}_5(y) = y^5-6y^4+19y^3-34y^2+33y-13$ (see~(\ref{eq:C_5(y)})).

Let $G' = G - V(F_2)$ and let $S_1$ and $S_5$ be the coloured neighbourhoods of $S_1$ and $S_5$ respectively.
Let $G_1 = G' / S_1$, $G_5 = G' / S_5$ and $G_{15} = G' / (S_1 \cup S_5)$.
Then the number of $(y+1)$-colourings of $G'$ in which neither $S_1$ nor $S_5$ is monochromatic is $Q_{G'}(y) - Q_{G_1}(y) - Q_{G_5}(y) + Q_{G_{15}}(y)$.
The number of $(y+1)$-colourings of $G'$ in which $S_1$ (respectively, $S_5$) but not $S_5$ (respectively, $S_1$) is monochromatic is $Q_{G_1}(y) - Q_{G_{15}}(y)$ (respectively, $Q_{G_5}(y) - Q_{G_{15}}(y)$).
Now the number of ways to colour $G$ is at most
\begin{align*}
&(y^5 - 4y^4 + 10y^3 - 13y^2 +10y - 3)Q_{G_{15}} \\
&\qquad+ (y^5 - 5y^4 + 14y^3 - 22y^2 + 20y - 8)(Q_{G_1}(y) - Q_{G_{15}}(y)) \\
&\qquad+ (y^5 - 5y^4 + 14y^3 - 22y^2 + 20y - 8)(Q_{G_5}(y) - Q_{G_{15}}(y)) \\
&\qquad+ (y^5 - 6y^4 + 19y^3 - 34y^2 + 33y - 13)(Q_{G'}(y) - Q_{G_1}(y) - Q_{G_5}(y) + Q_{G_{15}}(y)) \\
&= (y^3 - 4y^2 + 8y - 6)Q_{G_{15}} + (y^4 - 5y^3 + 13y^2 - 18y + 11)Q_{G_1}(y) \\
&\qquad+ (y^4 - 5y^3 + 13y^2 - 18y + 11)Q_{G_5}(y) + (y^5 - 6y^4 + 19y^3 - 34y^2 + 33y - 13)Q_{G'}(y) \\
&\leq (y+1)y^{n-8}(y-1)((y^3 - 4y^2 + 8y - 6) + (y^4 - 5y^3 + 13y^2 - 18y + 11) \\
&\qquad+ (y^4 - 5y^3 + 13y^2 - 18y + 11)
+ (y^5 - 6y^4 + 19y^3 - 34y^2 + 33y - 13)y) \\
&= (y+1)y^{n-8}(y-1)(y^6 - 6y^5 + 21y^4 - 43y^3 + 55y^2 - 41y + 16) \\
&< (y+1)y^{n-3}(y-1)(y-2),
\end{align*}
where the first inequality follows from Lemma~\ref{lem:subgraph} applied with $F = F_1$ and $G = G'$ and the last inequality holds since when the substitution $y = z + 3$ is made, the difference of the RHS and LHS has only positive coefficients.

\medskip

\Case{4d}{$F_2$ is a copy of $C_5$.}

We apply Proposition \ref{prop:C3,C5,K23}(c) to conclude that the number of extensions is at most
$y^5 - 5y^4 + 14y^3 - 23y^2 + 21y - 9 \leq y^5 - 5y^4 + 15y^3 - 28y^2 + 31y - 16 < y^4(y-2)$,
where the inequality holds since when the substitution $y = z + 3$ is made, the difference of the RHS and LHS has only positive coefficients.
It follows immediately that $Q_G(y) < (y+1)y^{n-3}(y-1)(y-2)$.

\medskip

\Case{5}{$G$ contains a $2$-induced $5$-cycle, a 2-induced $7$-cycle or a 2-induced $P_6$.}

We may assume that $G$ is $K_3$-free. Let $F$ be the considered 2-induced subgraph isomorphic to $C_5$, $C_7$, or $P_6$.
Since $F$ is $2$-induced, each vertex of $G - V(F)$ has at most one neighbour in $F$.
Let $C$ be any component of $G - V(F)$.
By Claim \ref{claim:twoInduced}, $C$ has fewer than $y^{|C|}$ $(y+1)$-colourings.
We colour each component and we have at most $y^{|G|-|F|}$ colourings. If $F=C_7$, we also colour one of the vertices of $F$ (for each colouring, there are at most $y$ ways to do so) and now we consider the remaining induced path $P_6$ as $F$.

Now, if $F=P_6$, we proceed in a similar way as in Case 4c, but in each of the auxiliary graphs $G'' \in \{G_1, G_5, G_{15}\}$ we identify only a single pair of vertices from the appropriate neighbourhood.
If these vertices are in different components of $G - V(F)$ then $G''$ still has at most $y^{|G'|}$ colourings.
If the vertices are in the same component $C$ of $G - V(F)$ then the resulting component $C'$ remains connected and contains a cycle, so it has at most $y^{|C'|} + y^{|C'| - 3}$ (with equality if and only if $C'$ is unicyclic and contains a $4$-cycle).
Therefore we obtain that
\begin{align*}
Q_G(y) &\le y^{n-9}(y^3 + 1)(y^6 - 6y^5 + 21y^4 - 45y^3 + 55y^2 - 41y + 16) \\
&< (y+1)y^{n-3}(y-1)(y-2).
\end{align*}

If $F=C_5$, then we proceed in the same way as in Case 4d to obtain:
$$
   Q_G(y) \le y^{n-5}(y^5 - 5y^4 + 15y^3 - 28y^2 + 31y - 16) < (y+1)y^{n-3}(y-1)(y-2),
$$
where in each case the last inequality holds since when the substitution $y = z + 3$ is made, the difference of the RHS and LHS has only positive coefficients.
\medskip

\Case{6}{None of the previous cases apply.}

Let $F'$ be a shortest odd cycle in $G$; then $F'$ is an induced cycle in $G$.
We may assume that $F'$ has length at least $5$, since $G$ is $K_3$-free and that is not $2$-induced (since otherwise we would have a 2-induced $C_5$, $C_7$ or $P_6$).
Let $v$ be a vertex with two neighbours on $F'$.
Then $v$ cannot have three or more neighbours on $F'$, or two neighbours whose distance along $F'$ is greater than $2$,
or two adjacent neighbours;
otherwise, there would be a shorter odd cycle.
So $v$ has two neighbours on $F'$ at distance exactly $2$.
Let $v'$ be the vertex of $F'$ which is adjacent to both neighbours and let $F = G[V(F') \cup \{v\}]$.
Let $v_1, \ldots, v_{|F'|-1}$ be the vertices of the path $F'-v'$ in the natural order.

We claim that $|V(G) \sm V(F)| \geq 5$.
Indeed, if $|F'| = 5$ then this follows from the fact that a $K_3$-free $4$-chromatic graph must have at least $11$ vertices \cite{WuZhang08}.
On the other hand, if $|F'| \geq 11$ then this follows from the fact that all but two vertices of $F'$ have a neighbour outside $F$, while any vertex outside $F$ has at most two neighbours on $F'$.
This leaves the cases $|F'| = 7$ and $|F'| = 9$.
Assuming that there are at most 4 vertices outside $F$, there are at most 8 edges between $F$ and $G-V(F)$.
In the case $|F'| = 9$, each vertex of degree $2$ in $F$ has exactly one neighbour outside $F$ (and the vertices of degree $3$ in $F$ have none).
The only way to have pairs of these edges to the four vertices outside is that $v_3$ and $v_5$ have a common neighbour and $v_4$ and $v_6$ have a common neighbour outside $F$.
We colour $v_1$, $v_3$ and $v_5$ with colour $1$, $v_4$, $v_6$ and $v_8$ with colour $2$ and the remaining vertices of $F$ with colour $3$.
Now some vertex of $G - V(F)$ must be adjacent to $v_2$ and either $v$ or $v'$, another to $v_7$ and the other of $v$ and $v'$, a third to $v_3$ and $v_5$ and the last to $v_4$ and $v_6$.
Since any $K_3$-free graph on $4$ vertices is either $C_4$ or a forest, and hence is $2$-choosable, there is a $3$-colouring of $G$, a contradiction.

In the case $|F'| = 7$, consider the $3$-colouring of $F$ in which $v_1$ and $v_3$ have colour $1$, $v_4$ and $v_6$ have colour $2$ and the remaining vertices of $F$ have colour $3$.
There are distinct vertices $w_1$ and $w_2$ of $G - V(F)$ which are adjacent to $v$ but not $v'$, and to $v'$ but not $v$, respectively.
Clearly, $w_1$ and $w_2$ cannot be adjacent (since this would create a $5$-cycle).
Further there are other distinct vertices $w_3$ and $w_4$ which are adjacent to $v_3$ and $v_4$ respectively.
If $w_3$ is not adjacent to a vertex of $F$ which does not have colour $1$, we can extend our $3$-colouring to $G - V(F)$;
this is trivial if $G - V(F)$ is a forest and if it is a $4$-cycle then we colour $w_3$ with colour $3$, $w_4$ with any available colour and $w_1$ and $w_2$ with colour $1$ or $2$ (whichever was not chosen for $w_4$).
Hence $w_3$ is adjacent to $v_5$; similarly, $w_4$ is adjacent to $v_2$.
But now neither $w_1$ nor $w_2$ can be adjacent to $w_3$ or $w_4$, or we would create a $5$-cycle.
From here extending the $3$-colouring to $G - V(F)$ is trivial.
This completes the argument showing that $|G|-|F|\ge5$.

We now observe that the number of ways to $(y+1)$-colour the vertices of $F$ is less than $y^{|F|} - y^{|F| - 1} + y^{|F| - 2}$.
We colour $F$, and then recursively colour remaining vertices one by one, such that each vertex taken has an already-coloured neighbour, until exactly five vertices remain.
Let $F_2$ be the subgraph of $G$ induced on the set of remaining vertices.

\medskip

\Case{6a}{$F_2$ is a graph formed from $C_4$ by appending a leaf.}
By a similar argument as in Case 4a, the number of ways to colour $G$ is less than
\begin{align*}
&(y^5 - 5y^4 + 14y^3 - 23y^2 + 23y - 11)(y^{n-5} - y^{n-6} + y^{n-7}) \\
&\qquad+ (y^4 - 4y^3 + 10y^2 - 13y + 8)y^{n-6} \\
&= (y^7 - 6 y^6 + 21 y^5 - 46 y^4 + 70 y^3 - 70 y^2 + 42 y - 11)y^{n - 7} < (y+1)y^{n-3}(y-1)(y-2).
\end{align*}
where the inequality holds since when the substitution $y = z + 3$ is made, the difference of the RHS and LHS has only positive coefficients.

\medskip

\Case{6b}{$F_2$ is isomorphic to $C_5$.}
By the same arguments as in Case 4d, we have that the number of ways of extending the colouring to $F_2$ is less than $y^4(y-2)$.
It follows immediately that $Q_G(y) < (y^2 - y + 1)y^{n-3}(y-2) < (y+1)y^{n-3}(y-1)(y-2)$.

\medskip

\Case{6c}{Neither of the previous cases apply.}
Since $F_2$ is triangle-free and not isomorphic to $C_5$, it has an independent set $S$ of size $3$.
We aim to colour two of the vertices of $F_2$, each with an already-coloured neighbour, to leave an independent set $F_3$ of size $3$.

If the two vertices of $F_2 - S$ are adjacent then one of them has degree at most $2$ in $F_2$, since $F_2$ is $K_3$-free. Since $\delta(G)\ge3$, this means that the vertex has an already coloured neighbour.
In this case we colour first the vertex of degree at most $2$ in $F_2$, and then the other, leaving $F_3 = S$.
So we may assume that the two vertices of $F_2 - S$ are not adjacent.
If both have degree at most $2$ in $F_2$, then we colour each of them, leaving $F_3 = S$. If both have degree 3 in $F_2$, then each of them has a neighbour outside $S$ since in a critical graph, the neighbourhood of a vertex cannot be a subset of the neighbourhood of another vertex. By Case 6a, the only remaining possibility is that
one of them, $w$, is adjacent to every vertex of $S$ and the other, $u$, has at most one neighbour, $v$, in $S$.
If $u$ has no neighbour in $S$ then we let $v$ be an arbitrary vertex of $S$.
Now we colour first $v$ and then $w$, leaving $F_3 = (S \cup \{u\}) \sm \{v\}$.

We label the vertices of $F_3$ as $\{v_1, v_2, v_3\}$.
For each $I \subseteq [3]$, let $Q_I$ be the number of colourings of $G - V(F_3)$ in which $N(v_i)$ is monochromatic if and only if $i \in I$.
Let $Q'_I$ be the number of $(y+1)$-colourings of $G - V(F_3)$ in which $N(v_i)$ is monochromatic for each $i \in I$.
Then $Q'_{\emptyset} = Q_{G - V(F_3)}(y) \leq (y^2 - y + 1)y^{n-5}$, while in general $Q'_I$ is the number of $(y+1)$-colourings of the graph formed from $G - V(F_3)$ by identifying $N(v_i)$ into a single vertex, for each $i \in I$.
Since each such graph is connected and contains an odd cycle, $Q'_{I} \leq y^{n-5}$ whenever $I \ne \emptyset$.
Since in a $4$-critical graph no two vertices can be twins, $Q'_{I} \leq y^{n-6}$ for $|I| \geq 2$.

Now the total number of $(y+1)$-colourings of $G$ is at most
\begin{align*}
&(y-1)^3 Q_{\emptyset} + y(y-1)^2 (Q_1 + Q_2 + Q_3) + y^2 (y-1) (Q_{12} + Q_{13} + Q_{23}) + y^3 Q_{123} \\
&= (y-1)^3 Q'_{\emptyset} + (y-1)^2 (Q'_1 + Q'_2 + Q'_3) + (y-1) (Q'_{12} + Q'_{13} + Q'_{23}) + Q'_{123} \\
&\leq (y-1)^3 (y^2 - y + 1) y^{n-5} + 3(y-1)^2 y^{n-5} + 3(y-1) y^{n-6} + y^{n-6} \\
&= (y^6 - 4 y^5 + 7 y^4 - 4 y^3 - 2 y^2 + 5 y - 2) y^{n - 6} < (y+1)y^{n-3}(y-1)(y-2),
\end{align*}
where the last inequality holds since when the substitution $y = z + 3$ is made, the difference of the RHS and LHS has only positive coefficients; and the equality between the first and the second line is the consequence of the identities
$$
    Q'_I = \sum_{J\supseteq I} Q_J \qquad (I\subseteq [3]).
$$
This completes the proof of Theorem \ref{thm:main}.

\section{Conclusion}

Our main result confirms that every 4-chromatic graph $G$ satisfies (\ref{eq:3}) for every positive integer $x$. It is not clear whether the same inequality holds for noninteger values of $x\ge4$; we would be tempted to believe so, but we see no apparent reasons why this would be true. By analysing tight cases in our proof of Theorem \ref{thm:main}, it is evident that in order to prove a statement about nonintegral values, the critical case would likely be when $x$ is close to $4$.

As we show in \cite{KM3}, validity of (\ref{eq:1}) is relatively easy to establish when $G$ has a critical subgraph whose order is much larger than $\chi(G)$.
The proofs used in this paper already indicate that the most important property for bounding the number of colourings is having large minimum degree. The following strengthening of Theorem \ref{thm:main} will be derived in \cite{KM3}:
Every sufficiently large $n$-vertex graph of minimum degree at least\/ $3$ and no twin vertices satisfies:
\begin{equation}\label{eq:strengthening}
    Q_G(y) \le y^{n - c\log n} (y-1)^{c\log n}
\end{equation}
for every integer $y \ge 3$, where $c>0$ is a constant. Observe that (\ref{eq:strengthening}) does not hold for the complete bipartite graph $K_{3,n-3}$, for which $Q_{K_{3,n-3}}(y) > y^{n-2}$, so the condition on excluding twins is needed.

Although we could outline the proof of the above-mentioned result based on arguments used in this paper, the proof is deferred to \cite{KM3}, where we prove a more general version of this statement.

\bibliographystyle{abbrv}

\bibliography{Tomescu_biblio}

\end{document}